\documentclass[12pt,twoside,leqno,final]{amsart}
\usepackage{amsmath}
\usepackage{amssymb}
\usepackage{graphics} 

\theoremstyle{plain}
\newtheorem{thm}{Theorem}[section]
\newtheorem{lem}[thm]{Lemma}

\newtheorem{cor}[thm]{Corollary}
\newtheorem{rem}{Remark}

\newcommand{\R}{{\mathbb R}}
\newcommand{\Z}{{\mathbb Z}}

\title[Two weight norm inequalities]{Two weight norm inequalities for  vector-valued operators}

\begin{document}
\author{Carme Cascante and Joaquin M. Ortega}
\address{Carme Cascante and Joaquin M. Ortega: Dept.\ Matem\`atica Aplicada i An\`alisi, Universitat  de Barcelona, Gran Via 585, 08071 Barcelona, Spain}
\email{cascante@ub.edu,\quad ortega@ub.edu}

\date{\today}
\thanks{Partially supported by  DGICYT Grant MTM2011-27932-C02-01  and DURSI Grant 2014SGR 289.}

\begin{abstract} 
We study  two weight norm inequalities for a vector-valued operator  from a weighted $L^p(\sigma)$-space to mixed norm $L^q_{l^s}(\mu)$ spaces, $1<q<p$. We apply these results to the boundedness of Wolff's potentials.
\end{abstract}
\maketitle
  \section{Introduction}  

The object of this paper is the study of two weight norm inequalities for a vector-valued operator  from a weighted $L^p(\sigma)$-space to mixed norm $L^q_{l^s}(\mu)$ spaces, $1<q<p$. More precisely, we study the following problem:
Let ${\mathcal D}$ be the standard dyadic system in $\R^n$, given by ${\mathcal D}:=\{ 2^{-k}([0,1)^n+j);\, k\in\Z,\, j\in \Z^n\}$,   $(\lambda_Q)_{Q\in {\mathcal D}}$ a sequence of nonnegative real numbers, let $T$ be the operator defined by
$$T(f)=\left( \lambda_Q\left(\int_Q fd\sigma\right)\chi_Q\right)_{Q\in {\mathcal D}},$$
where $\chi_Q$ denotes the characteristic function of $Q$. Given $1\leq s\leq +\infty$, $1<p,q<+\infty$, which are the pair of positive Borel measures $\mu$, $\sigma$ on $\R^n$ such that for any nonnegative function $f$,
 \begin{equation}\label{problem2}
  \|T(f)\|_{L^q_{l^s}(\mu)}=\|\left(\sum_{Q\in {\mathcal D}}\lambda_Q^s \left( \int_Q fd\sigma \right)^s \chi_Q\right)^\frac1{s} \|_{L^q(\mu)}\leq C \|f \|_{L^p(\sigma)}?
 \end{equation}
The possible characterizations of \eqref{problem2} depend heavily on the size of the parameters $s,p$ and $q$ involved.

For the range $1<p\leq q<\infty$, the problem is rather well understood.

If $s=1$, the operator ${\displaystyle f\rightarrow \left(\sum_{Q\in {\mathcal D}}\lambda_Q^s \left( \int_Q fd\sigma \right)^s\chi_Q\right)^{1/s}}$ is linear and its boundedness is characterized by the Sawyer testing condition (testing the boundednes of the operator and its adjoint against characteristic functions of cubes). See  \cite{nazarovtreilvolberg}, \cite{laceysawyeruriartetuero} \cite{treil} and \cite{hytonen}.

If $s=\infty$, the boundedness is characterized by the direct Sawyer condition \cite{sawyer}.

If $1<s<\infty$ and $s\geq p$ it was characterized also by the direct Sawyer's conditions in \cite{cascanteortega} and for $s<p$ it was reduced to the case $s=1$.

 The case $1<s<\infty$ and $q\geq p$  was  considered by J. Scurry in \cite{scurry}, extending M.T. Lacey, E.T. Sawyer and I. Uriarte-Tuero's proof of the case $s=1$. For $p=q$,  a different approach, which is closer to the methods used in the proof of one of our main results, was considered by T.S. H\"anninen in \cite{hanninen}, adapting T.P. Hyt\"onen's proof for the case $s=1$ (see \cite{hytonen}), based on a parallel stopping cubes method. Namely, he proved that if $q=p$ and $s>1$, and if we denote ${\displaystyle T_P(f):=  \left( \lambda_Q \left(\int_Q fd\sigma\right)\chi_Q\right)_{\substack{ Q\subset P,\\Q\in {\mathcal D}}}}$ and $T^*_P$ its formal adjoint defined by $ {\displaystyle T^*_P(g):= \sum_{\substack{ Q\subset P,\\Q\in {\mathcal D}}}}\lambda_Q \left(\int_Q g_Q d\mu \right)\chi_Q$, then
 \eqref{problem2} holds if and only if both of the following conditions are satisfied:
\begin{enumerate}
\item For all $P\in{\mathcal D}$, ${\displaystyle\|T_P(\sigma)\|_{L_{l^s}^p(\mu)}\leq C \sigma(P)^{\frac1{p}}}$.
\item For all $g=(g_Q\chi_Q)_{Q\in{\mathcal D}}$, $g_Q\geq 0$, $${\displaystyle\|T_P^*(gd\mu)\|_{L^{p'}(\sigma)} \leq C \|g\|_{L_{l^{s'}}^\infty(\mu)}\mu(P)^\frac1{p'}}.$$
\end{enumerate}

The approach was further extended to an abstract Banach valued setting in \cite{hanninen1}.

The upper triangle case $q<p$ in the linear case $s=1$ was considered in \cite{cascanteortegaverbitsky} assuming some extra conditions on the operator and the measure $\sigma$, conditions that were removed  by H. Tanaka in  \cite{tanaka}. We also recall that in  \cite{hanninenhytonenli} (see also \cite{vuorinen}) was obtained a unified  characterization which do not depend in the relative position of $p$ and $q$.
Namely, for the particular case that $q<p$, the characterization obtained by T.S. H\"anninen, T.P. Hyt\"onen and K. Li in  \cite{hanninenhytonenli} was the following:
The operator $T$ is bounded from $L^p(\sigma)$ to  $L^q(\mu)$ if and only if the following two conditions hold:

\begin{enumerate}
\item  $$\sup_{\mathcal F} \left\| \left\{  \frac{\|T_F(\sigma)\|_{L^q(\mu)}}{\sigma(F)^{\frac1{p}}}\right\}_{F\in{\mathcal F}} \right\|_{l^r({\mathcal F})}\leq C,$$
\item $$\sup_{\mathcal G} \left\| \left\{  \frac{\|T_G^*(\mu)\|_{L^{p'}(\sigma)}}{\mu(G)^{\frac1{q'}}}\right\}_{G\in{\mathcal G}} \right\|_{l^r({\mathcal G})}\leq C,$$
\end{enumerate}
where the supremums are taken over all subcollections ${\mathcal F}$ and ${\mathcal G}$ of ${\mathcal D}$ that are sparse (in the sense of the definition given below) with respect to $\sigma$ and $\mu$ respectively and $1/r=1/q-1/p$ and the constants are independent of ${\mathcal F}$ and ${\mathcal G}$. In the same paper it is obtained a characterization for the maximal dyadic function ($s=\infty$) and $q<p$ (see also \cite{hanninenthesis}).

In this paper we will obtain two type of results. On one hand, the first characterization follows the approach of the works in \cite{hanninen} and \cite{hanninenhytonenli}. However, instead of the cases consider in \cite{hanninen} ($p=q$), we assume that $q<p$ and  unlike the case  in \cite{hanninenhytonenli} ($s=1$), the operator considered here maps $L^p(\sigma)$ to $L_{l^s}^q(\mu)$, where $s\neq 1$, situation that gives that $T$ and $T^*$ are not symmetric. On the other hand, our second characterization follows an approach which is, in some sense, more original. It uses a reduction to the case $q=1$ and a proof based on duality.

  We will see in the next section that the cases where $s\leq q<p$ can be reduced to  $s=1$ and, in consequence, to the linear case considered in \cite{tanaka}, \cite{cascanteortegaverbitsky3} and  in \cite{hanninenhytonenli}. Hence, we limit ourselves to the case where $q<p$ and $s>q$. Before we state the main results, we will need to introduce some notations. 
	
A family of dyadic cubes ${\mathcal F}$ is $\sigma$-sparse (or sparse with respect to $\sigma$), if for each $F\in{\mathcal F}$, there exists $E_F\subset F$ such that $\sigma(E_F)\geq (1/2) \sigma(F)$ and the sets $(E_F)_{F\in{\mathcal F}}$ are pairwise disjoints. Of course, the constant $1/2$ can be replaced by any other fixed constant $\delta
\in(0,1)$.

If ${\mathcal F}$ is a subcollection of ${\mathcal D}$ $\sigma$-sparse and  $Q\in {\mathcal D}$, we denote by
$$\pi_{\mathcal F} (Q)= \min \{ F\in {\mathcal F};\, Q\subset F\}$$
 and analogously for $\pi_{\mathcal G}$ where ${\mathcal G}$ is $\mu$-sparse. 

 Then, if $g=(g_Q)_{Q\in{\mathcal D}}$, $g_Q\in\R$, $G\in{\mathcal G}$,   we denote $g_G: =\{ g_Q\}_{\pi_{\mathcal G}(Q)=G}$. The formal adjoint operator of $T$ is denoted by $T^*$  and we define $r$ by $1/r=1/q-1/p$.

\begin{thm}\label{thm:qp}
Let $1<q<p<\infty$, $ q< s \leq \infty$, and let $\mu,\sigma$ be positive locally finite measures on $\R^n$, and $(\lambda_Q)_{Q\in {\mathcal D}}$, $\lambda_Q\geq 0$, $Q\in{\mathcal D}$. We then have:

\begin{equation}\label{eqn:qp}\|T(f)\|_{L^q_{l^s}(\mu)}=\|\left(\sum_{Q\in {\mathcal D}}\lambda_Q^s \left( \int_Q fd\sigma \right)^s \chi_Q\right)^\frac1{s} \|_{L^q(d\mu)}\leq C \|f \|_{L^p(\nu)}\end{equation} if and only if the following two conditions hold:

\begin{enumerate}
\item\label{enum:qp1} There exists $C_1>0$ such that for any ${\mathcal G}$ subcollection of dyadic cubes in $ {\mathcal D}$ $\mu$-sparse, we have 
$$\left\|\left\{  \frac{\| T^*(g_G)\|_{L^{p'}(\sigma)}}{\mu(G)^{\frac1{q'}}\|g_G\|_{L_{l^{s'}}^\infty(\mu)}} \right\}_{G\in{\mathcal G}}\right\|_{l^r({\mathcal G})}\leq C_1.$$

\item\label{enum:qp20} There exists $C_2>0$ such that 
for any ${\mathcal F}$ subcollection of dyadic cubes in $ {\mathcal D}$, $\sigma$-sparse, and any $(\beta_F)_{F\in{\mathcal F}}$, $\beta_F\geq 0$, 
$$\left\| \sum_{F\in{\mathcal F}}\frac{\beta_F}{\sigma(F)^{1/p}}  T(\chi_F) \right\|_{L_{l^s}^q(\mu)}\leq C_2 \left(  \sum_{F\in{\mathcal F}}\beta_F^p\right)^{1/p}.$$
\end{enumerate}
In addition, if $C, C_i$, $i=1,2$ are the smallest constants in \eqref{eqn:qp}, \eqref{enum:qp1} and \eqref{enum:qp20} respectively, we have that $C\approx C_1+C_2$.
\end{thm}

In our second main result, we need that the measure $\mu$ has no point masses.

\begin{thm}\label{thm:reduction}
Assume that $1<q<p$ and $q\leq s\leq \infty$. Let $\mu,\sigma$ be positive locally finite measures. Assume that, in addition, the measure $\mu$ has no point masses. Let $T(f)=(\lambda_Q \left(\int_Q fd\sigma\right)\chi_Q)_{Q\in {\mathcal D}} $. Then
\begin{equation}\label{eqn:qpduality}\| T(f)\|_{L_{l^s}^q(\mu)} \leq C||f||_{L^p(\sigma)},\end{equation}
 if and only if there exists $C_1>0$ such that for any 
$(E_Q)_{Q\in{\mathcal D}}$ pairwise disjoint $\mu$-measurable sets such that $E_Q\subset Q$, 

$$\left\| \sum_{Q\in{\mathcal D}} \frac{(\lambda_Q\sigma(Q))^q\mu(Q)^{1/{\tilde{s}}} \mu(E_Q)^{1/{\tilde s}'}}{\sigma(Q)}\chi_Q \right\|_{L^{{\tilde p}'}(\sigma)}\leq C_1, $$
where ${\tilde s}=s/q$ and ${\tilde p}=p/q$.

\end{thm}

The paper is organized as follows: 
In Section \ref{section1} we rewrite the estimate \eqref{problem2} as a  problem on discrete multipliers and introduce the definitions and some of the main lemmas required for the proof of our results. In Section \ref{section2} we give the proof of Theorem \ref{thm:qp} and we obtain characterizations for the weak boundedness of the operator $T$. In Section \ref{reduction}
 we give the proof of Theorem \ref{thm:reduction} and in the last Section we give an application of this last result to the boundedness of Wolff-type potentials.

\section{Some preliminaries }\label{section1}

The following lemma  is a consequence of the boundedness of the dyadic maximal function and it shows that (\ref{problem2}) can be rewritten in terms of discrete multipliers (see for example, Lemma 1 in \cite{verbitsky1} or Lemma 2.1 in 
\cite{cascanteortega} for a proof).
\begin{lem}\label{lemma1}
 Assume $1<p<+\infty$. Then estimate (\ref{problem2}) holds if and only if there exists $C>0$ such that for any sequence $(\rho_Q)_Q$ of nonnegative numbers,
 \begin{equation}\label{problem3}
  \left\|\left(\sum_{Q\in {\mathcal D}}\lambda_Q^s(\sigma(Q))^s \rho_Q^s \chi_Q\right)^\frac1{s} \right\|_{L^q(d\mu)} \leq C||\sup_{Q\in{\mathcal D}}(\rho_Q\chi_Q)||_{L^p(\sigma)}.
\end{equation}
  
 \end{lem} 

Before we prove our results, we will point out some observations concerning the cases where $s\leq q$ and $q<p$.   We observe that when $s\leq q$, if we write $t_Q=\rho^s$, and denote $\tilde{p}=p/s$ and $\tilde{q}=q/s$, (\ref{problem3}) can be rewritten as

  \begin{equation*}
  \left\|\sum_{Q\in {\mathcal D}}\lambda_Q^s(\sigma(Q))^s t_Q \chi_Q \right\|_{L^{\tilde{q}}(d\mu)} \leq C||\sup_{Q\in{\mathcal D}}(t_Q\chi_Q)||_{L^{\tilde{p}}(\sigma)}.
\end{equation*}
  
By Lemma \ref{lemma1}, the above estimate is equivalent to the boundedness from $L^{{\tilde p}}(\sigma)$ to $L^{{\tilde q}}(\mu)$ of the linear operator $${\displaystyle {\tilde T}(f)=\sum_{Q\in{\mathcal D}} \lambda_Q^s (\sigma(Q))^s \frac{\int_Qfd\sigma}{\sigma(Q)}\chi_Q},$$
 that is:
 \begin{equation}\label{problem5}\|{\tilde T}(f)\|_{L^{\tilde{q}}(\mu)}\lesssim \|f\|_{L^{\tilde{p}}(\sigma)}.\end{equation}
As we have already said in the introduction, if we assume in addition that $s<q$, that is $\tilde{q}>1$, this estimate was characterized in \cite{tanaka} (see also \cite{cascanteortegaverbitsky3}) and with a different approach in \cite{hanninenhytonenli}.

Next, if $s=q<p$, $\tilde{q}=1$ and consequently,
\begin{align*}
& \|{\tilde T}(f)\|_{L^1(\mu)}=\|\sum_{Q\in{\mathcal D}} \lambda_Q^s (\sigma(Q))^{s-1} \left(\int_Q f d\sigma\right) \chi_Q\|_{L^1(\mu)}=\\
&\sum_{Q\in{\mathcal D}} \lambda_Q^s (\sigma(Q))^{s-1} \int_Q fd\sigma \mu(Q)= \int_{\R^n}\sum_{Q\in{\mathcal D}} \lambda_Q^s  (\sigma(Q))^{s-1}\mu(Q)f\chi_Q d\sigma.
\end{align*}

Hence, duality gives now that \eqref{problem5} holds if and only if
$$
\sum_{Q\in{\mathcal D}} \lambda_Q^s (\sigma(Q))^{s-1} \mu(Q)\chi_Q \in L^{(p/q)'}(\sigma).$$

So we are left to deal with the case $q<p$ and $s> q$. In the rest of the paper we will assume that this is the situation.

In order to prove our  results, we need to introduce some definitions and recall some known facts.
For every ${\mathcal F}$  subcollection of dyadic cubes, we denote

$${\rm ch}_{\mathcal F}(F):=\{ {\rm maximal}\,F'\subset F:\, F'\neq F;\, F'\in {\mathcal F}\}$$ and $\displaystyle{ E_{\mathcal F}(F):= F\setminus \cup_{F'\in {\rm ch}_{\mathcal F}(F)}F'}$.

The sets $E_{\mathcal F}(F)$ are pairwise disjoint. We observe that if $\sigma(E_{\mathcal F}(F))\geq (1/2) \sigma(F)$, then the family ${\mathcal F}$ is $\sigma$-sparse.

It is important to recall that if a family ${\mathcal F}$ is $\sigma$-sparse, then ${\mathcal F}$ is $\sigma$-Carleson, in the sense that for any $P\in{\mathcal F}$, 
\begin{equation} \label{lem:Carleson}\sum_{Q\in{\mathcal F},\, Q\subset P} \sigma(Q)\leq 2\sigma(P).
\end{equation}

Indeed, if ${\mathcal F}$ is $\sigma$-sparse, for each $F\in{\mathcal F}$, there exists $E_F\subset F$ pairwise disjoint, such that $\sigma(E_F)\geq (1/2) \sigma(F)$. Then
$$\sum_{Q\in{\mathcal F},\, Q\subset P}\sigma(Q)\leq 2 \sum_{Q\in{\mathcal F},\, Q\subset P}\sigma(E(Q)) \leq 2\sigma(P).
$$
Nevertheless, the reciprocal is not true, in general, if the measure $\sigma$ has point masses, as we will see in Section \ref{reduction}.

We recall the well known dyadic Carleson embedding:
\begin{lem}\label{lem:Caremb}
Let $1<p<\infty$. If ${\mathcal F}$ is $\sigma$-Carleson, then
$$\left( \sum_{F\in{\mathcal F}} (\langle f\rangle_F^\sigma)^p\sigma(F) \right)^{1/p}\lesssim\|f\|_{L^p(\sigma)},$$
where ${\displaystyle \langle f\rangle_F^\sigma=\frac{\int_F fd\sigma}{\sigma(F)}}$.
\end{lem}

The next lemma, proved in  \cite{hanninenhytonen},  will be used in the proof of Theorem \ref{thm:qp}:
\begin{lem}\label{lem:2.3}
Let $1\leq p<\infty$. let $\sigma$ be a locally finite Borel measure. Let ${\mathcal F}$ be a $\sigma$-sparse collection of dyadic cubes. For each $F\in{\mathcal F}$, assume that $a_F$ is a non-negative function supported on $F$ and constant on each $F'\in {\rm ch}_{\mathcal F}\,(F)$. Then
\begin{align*}
&\left( \sum_{F\in{\mathcal F}} \|a_F\|_{L^p(\sigma)}^p\right)^{1/p}\\&
 \leq \|\sum_{F\in{\mathcal F}}a_F\|_{L^p(\sigma)}\leq 3p \left(\sum_{F\in{\mathcal F}} \|a_F\|_{L^p(\sigma)}^p\right)^{1/p}.
\end{align*}
\end{lem}

\section{Characterizations based on parallel stopping  cubes}\label{section2}

\subsection{Proof of Theoren \ref{thm:qp}}
\begin{proof}

We begin with the proof of the necessity of conditions \eqref{enum:qp1} and \eqref{enum:qp20} in Theorem \ref{thm:qp}. Assume that \eqref{eqn:qp} holds. By duality, this is equivalent to the norm inequality
\begin{equation*}
\|T^* (g)\|_{L^{p'}(\sigma)} \lesssim \|g\|_{L_{l^{s'}}^{q'}(\mu)}.
\end{equation*}

 Let us start with \eqref{enum:qp1}. We will follow some of the arguments in Proposition 3.1 in \cite{hanninenhytonenli}. If $g=(g_Q)_{Q\in{\mathcal D}}$ and ${\mathcal G}\subset {\mathcal D}$ is $\mu$-sparse, we define for every $G\in{\mathcal G}$, $\phi_G=1/(\mu(G)^{\frac1{q'}}\|g_G\|_{L_{l^{s'}}^\infty(\mu)})$. We want to show that
$$\left\{  \| T^*(\phi_G g_G)\|_{L^{p'}(\sigma)}\right\}_{G\in{\mathcal G}}$$  is in $l^r({\mathcal G})$. Applying duality with $(r/p')'=q'/p'>1$, we have
\begin{equation}\begin{split}\label{align:duality0}
&\left( \sum_{G\in{\mathcal G}}  \| T^*(\phi_G g_G)\|_{L^{p'}(\sigma)}^{r}\right)^{1/r}=\left( \sum_{G\in{\mathcal G}} ( \| T^*(\phi_G g_G)\|_{L^{p'}(\sigma)}^{p'})^{r/p'}\right)^{1/r}=\\
&=\sup \left\{ \left( \sum_{G\in{\mathcal G}}\alpha_G  \| T^*(\phi_G g_G)\|_{L^{p'}(\sigma)}^{p'}\right)^{1/p'};\,  \sum_{G\in{\mathcal G}}\alpha_G^{q'/p'}\leq 1 \right\}\\
&=\sup \left\{ \left( \sum_{G\in{\mathcal G}}  \| T^*(\beta_G\phi_G g_G)\|_{L^{p'}(\sigma)}^{p'}\right)^{1/p'};\,  \sum_{G\in{\mathcal G}}\beta_G^{q'}\leq 1 \right\} .
\end{split}\end{equation}
But since $p'>1$, the hypothesis \eqref{eqn:qp}, together with the triangle inequality for $l^{s'}$ and \eqref{align:duality0}, gives that
\begin{align*}
& \sum_{G\in{\mathcal G}}\| T^*(\beta_G\phi_G g_G)\|_{L^{p'}(\sigma)}^{p'}= \int_{\R^n}  \sum_{G\in{\mathcal G}}|T^* (\beta_G \phi_G g_G)|^{p'}d\sigma\\
&\leq \int_{\R^n}  \left|\sum_{G\in{\mathcal G}}T^* (\beta_G \phi_G g_G)\right|^{p'}d\sigma=\| T^* (\sum_{G\in{\mathcal G}}(\beta_G \phi_G g_G))\|_{L^{p'}(\sigma)}^{p'}\\
&\lesssim \| \sum_{G\in{\mathcal G}}(\beta_G \phi_G g_G)\|_{L_{l^{s'}}^{q'}(\mu)}^{p'}
\leq \| \sum_{G\in{\mathcal G}}(\beta_G \phi_G |g_G|_{{s'}})\|_{L^{q'}(\mu)}^{p'}.
\end{align*}
Here $|g_G|_{s'}=\left(\sum_{\pi_{\mathcal G}(Q)=G} g_Q^{s'}\chi_Q\right)^{1/s'}$.

Next,  observe that if $G\in{\mathcal G}$, and $\pi_{\mathcal G}(Q)=G$, then for any $G'\in {\rm ch}(G)$, we have that either $Q$ and $G'$ are pairwise disjoints or $G' \subset Q $ and $\beta_G \Phi_G|g_G|_{{s'}}$ is constant on $G'$. Hence, applying Lemma \ref{lem:2.3}, \eqref{align:duality0} is bounded by
\begin{align*}
&\left( \sum_{G\in{\mathcal G}} \|\beta_G \phi_G |g_G|_{s'}\|_{L^{q'}(\mu)}^{q'}\right)^{1/q'}= 
\left( \sum_{G\in{\mathcal G}} \beta_G^{q'} \phi_G^{q'} \int_{\R^n} |g_G|_{s'} ^{q'}d\mu\right)^{1/q'}\\
&\leq \left( \sum_{G\in{\mathcal G}} \beta_G^{q'} \phi_G^{q'} \|g_G\|_{L_{l^{s'}}^\infty(\mu)}^{q'}\mu(G)\right)^{1/q'}\lesssim 1.
\end{align*}

Now we check  \eqref{enum:qp20}. Let ${\mathcal F}$ be $\sigma$-sparse, and $(\beta_F)_{F\in{\mathcal F}}\in l^p({\mathcal F})$, $\beta_F\geq0$. Applying \eqref{eqn:qp} to the function ${\displaystyle \sum_{F\in{\mathcal F}} \frac{\beta_F}{\sigma(F)^{1/p}} \chi_F}$, together with Lemma \ref{lem:2.3}, we obtain:
\begin{align*}
&\|\sum_{F\in{\mathcal F}}  \frac{\beta_F}{\sigma(F)^{1/p}}T( \chi_F)\|_{L_{l^s}^q(\mu)}=\|T\left(\sum_{F\in{\mathcal F}}  \frac{\beta_F}{\sigma(F)^{1/p}}  \chi_F\right)\|_{L_{l^s}^q(\mu)}\\
&\lesssim \|\sum_{F\in{\mathcal F}} \frac{\beta_F}{\sigma(F)^{1/p}}  \chi_F\|_{L^p(\sigma)} \simeq \left( \sum_{F\in{\mathcal F}} \|\frac{\beta_F}{\sigma(F)^{1/p}}\chi_F\|_{L^p(\sigma)}^p\right)^{1/p}\\
&\simeq \left( \sum_{F\in {\mathcal F}} \beta_F^p\right)^{1/p}.
\end{align*}
Then \eqref{enum:qp20} holds.

Assume now that \eqref{enum:qp1} and \eqref{enum:qp20} hold. By duality, we want to show that
\begin{equation*}
\sum_{Q\in{\mathcal D}} \lambda_Q \int_Q fd\sigma \int_Q g_Qd\mu \lesssim \|f\|_{L^p(\sigma)}\|g\|_{L_{l^{s'}}^{q'}(\mu)}.
\end{equation*}

 As it is remarked in \cite{hanninen} and \cite{hanninenhytonenli}, by an application of Stein's inequality,  we may assume that  $g=(g_Q\chi_Q)_{Q\in {\mathcal D}}$, where $g_Q\in R$ and $f$ and $g_Q$ are nonnegative. We also may assume that the collection ${\mathcal D}$ is finite and that for some $Q_0\in{\mathcal D}$, we have that for all $Q\in{\mathcal D}$, $Q\subset Q_0$.
We will also use the stopping cubes defined for each of the pairs $(f,\sigma)$ and $(g,\mu)$, defined in \cite{hanninen}, considering the same subcollections ${\mathcal F}$, $\sigma$-sparse and ${\mathcal G}$, $\mu$-sparse .

For $F\in{\mathcal F}$, $G\in{\mathcal G}$, we write $\pi(Q)=(F,G)$ if $\pi_{\mathcal F}(Q)=F$ and $\pi_{\mathcal G}(Q)=G$.

We then have
$$
 \sum_{Q\in{\mathcal D}} \lambda_Q  \int_Qfd\sigma \int_Q g_Q d\mu\leq
I+J$$
where
\begin{equation}\label{eqn:split1}
I=\sum_{G\in{\mathcal G}} \sum_{\substack{F\in{\mathcal F};\\ F\subset G}} \sum_{\substack{Q\in{\mathcal D}\\ \pi(Q)=(F,G)}}\lambda_Q \int_Q g_Q d\mu \int_Q fd\sigma,
\end{equation}
and
\begin{equation}\label{eqn:split2}
II=\sum_{F\in{\mathcal F}} \sum_{\substack{G\in{\mathcal G};\\ G\subset F}} \sum_{\substack{Q\in{\mathcal D}\\ \pi(Q)=(F,G)}}\lambda_Q \int_Q g_Q d\mu \int_Q fd\sigma.
\end{equation}

We will estimate both sums \eqref{eqn:split1} and \eqref{eqn:split2} separately.
For the estimate $I$, we proceed as in \cite{hanninen} where the function $f$ was replaced by  functions $f_G$ such that 
\begin{equation}\label{eqn:split4}
\left( \sum_{G\in{\mathcal G}} \|f_G\|_{L^p(\sigma)}^p \right)^{1/p}\lesssim \|f\|_{L^p(\sigma)} .
\end{equation}
and ${\mathcal G}$ is such that for $G\in{\mathcal G}$, $\|g_G\|_{L_{l^{s'}}^\infty(\mu)} \leq 2 \langle |g|_{s'}\rangle_G^\mu$.

Then by H\"older's inequality,

$$I
\leq \sum_{G\in{\mathcal G}}  \|f_G\|_{L^p(\sigma)} \|T^*(g_G\mu)\|_{L^{p'}(\sigma)}$$
Since $1/p+1/r+1/q'=1$, applying again H\"older's inequality, we have that the above is bounded by
$$
\left( \sum_{G\in{\mathcal G}}  \|f_G\|_{L^p(\sigma)}^p\right)^{1/p} \left( 
 \sum_{G\in{\mathcal G}} \left(
 \frac{\|T^*(g_G\mu)\|_{L^{p'}(\sigma)}}{\mu(G)^{1/q'} 
\|g_G\|_{L_{l^{s'}}^\infty(\mu)} }
\right)^r
\right)^{1/r} \left( \sum_{G\in{\mathcal G}}\mu(G) \left( \|g_G\|_{L_{l^{s'}}^\infty(\mu)}\right)^{q'}  \right)^{1/q'} 
$$
Next, by \eqref{eqn:split4}, the above is bounded by
\begin{align*}& 
\|f\|_{L^p(\sigma)}\left( 
 \sum_{G\in{\mathcal G}} \left(
 \frac{\|T^*(g_G\mu)\|_{L^{p'}(\sigma)}}{\mu(G)^{1/q'} 
\|g_G\|_{{l^{s'}}^\infty{(\mu)}}}
\right)^r
\right)^{1/r} \left( \sum_{G\in{\mathcal G}}\mu(G) \left( \langle |g|_{s'}\rangle_G^\mu\right)^{q'}  \right)^{1/q'}. 
\end{align*}
Applying the hypothesis \eqref{enum:qp1} of Theorem \ref{thm:qp} and Lemma \ref{lem:Caremb} (recall that by \eqref{lem:Carleson}, ${\mathcal G}$  satisfies a $\mu$-Carleson condition), we finally obtain that $I\lesssim \|f\|_{L^p(\sigma)}\|g\|_{L_{l^{s'}}^{q'}(\mu)}$.

Now we estimate $II$. If we denote $g_F:=(g_Q)_{\substack{Q\in{\mathcal D};\\ \pi(Q)=(F,G) }}$ for some $G\in {\mathcal G}$ such that $G\subset F$, we argue as in \cite{hanninen} to obtain that
\begin{align*}&II\leq 
2 \sum_{F\in{\mathcal F}}\langle f\rangle_F^\sigma \sum_{\substack{G\in{\mathcal G}
;\\ G\subset F}} \sum_{\substack{Q\in{\mathcal D}\\ \pi(Q)=(F,G)}}\lambda_Q \sigma(Q) \int_Qg_Qd\mu\\&
=2 \sum_{F\in{\mathcal F}}\langle f\rangle_F^\sigma\int_{\R^n} \sum_{Q\in{\mathcal D}} (T_F(\sigma))_Q (g_F)_Q d\mu\\&
\lesssim \left\{ \sum_{F\in{\mathcal F}} \left( \langle f\rangle_F^\sigma\right)^p\sigma(F)\right)^{1/p}  \left\{ \sum_{F\in{\mathcal F}}\left( \int_{\R^n} \sum_{Q\in{\mathcal D}} \frac1{\sigma(F)^{1/p}} (T_F(\sigma))_Q (g_F)_Q d\mu\right)^{p'} \right\}^{1/p'}\\& \lesssim \|f\|_{L^p(\sigma)}  \left\{ \sum_{F\in{\mathcal F}}\left( \int_{\R^n} \sum_{Q\in{\mathcal D}} \frac1{\sigma(F)^{1/p}} (T_F(\sigma))_Q (g_F)_Q d\mu\right)^{p'} \right\}^{1/p'},
\end{align*}
and where in the last estimate we have used the $\sigma$-Carleson condition for the family ${\mathcal F}$.

Now, duality, H\"older's inequality and the hypothesis \eqref{enum:qp20} of Theorem \ref{thm:qp}, give
\begin{align*}& \left\| \left( \int_{\R^n} \sum_{Q\in{\mathcal D}} \frac1{\sigma(F)^{1/p'}} (T_F(\sigma))_Q (g_F)_Qd\mu\right)_{F\in{\mathcal F}}\right\|_{l^{p'}({\mathcal F})}\\& =
\sup_{\sum_{F\in{\mathcal F}}\beta_F^p\leq 1} \sum_{F\in{\mathcal F}}\int_{\R^n} \sum_{Q\in{\mathcal D}} \frac{\beta_F}{\sigma(F)^{1/p}} (T_F(\sigma))_Q (g_F)_Q d\mu \\&\lesssim 
\sup_{\sum_{F\in{\mathcal F}}\beta_F^p\leq 1} \left\| \sum_{F\in{\mathcal F}}\frac{\beta_F}{\sigma(F)^{1/p}}
 T_F(\sigma) \right\|_{L_{l^s}^q(\mu)}\|g\|_{L_{l^{s'}}^{q'}(\mu)} \lesssim \|g\|_{L_{l^{s'}}^{q'}(\mu)}.
\end{align*}
Of course the equivalence $C\simeq C_1+C_2$ follows from the last estimates.
\end{proof}
Observe that here we have used that $T_F(\sigma)\leq T(\chi_F)$. In fact, in condition \eqref{enum:qp20} of Theorem \ref{thm:qp}, we can substitute $T(\chi_F)$ by $T_F(\sigma)$.

\subsection{On weak estimates}

One natural question that arises  is the study of the weak boundedness of the operator 
\begin{equation}\label{eqn:weak0}T:L^p(\sigma)\to L_{l^s}^{q,\infty}(\mu).\end{equation} 

Here the space $L_{l^s}^{q,\infty}(\mu)$ consists of sequences of functions $f=(f_Q)_{Q\in{\mathcal D}}$ for which the function $\left(\sum_{Q\in {\mathcal D}} |f_Q|^s\right)^{1/s}\in L^{q,\infty }(\mu)$. We recall that  $g\in L^{q,\infty}(\mu)$ 
 if and only if, $$\|g\|_{L^{q,\infty}(\mu)}=\sup_{\lambda>0} \lambda^q \mu(\{x\in{\R}^n;\, |g(x)|>\lambda\})<+\infty.$$ 
The so called Kolmogorov's condition (see Lemma 2.8, Chapter V in \cite{kolmogorov}) gives an equivalent definition. Namely, $g\in L^{q,\infty}(\mu)$ if and only if, there exists $\alpha<q$ and $C>0$ such that for any measurable subset $E\subset {\R}^n$, such that $\mu(E)>0$,
\begin{equation*}
 \mu(E)^{\frac{\alpha-q}{\alpha q}}\left\{\int_E   |g(x)|^\alpha d\mu(x)\right\}^\frac1{\alpha}\leq C.\end{equation*}

We then have that $T:L^p(\sigma)\to L_{l^s}^{q,\infty}(\mu)$ if and only if, for some $\alpha$, $1<\alpha <q$ and any $E\subset \R^n$, with $\mu(E)>0$, 
\begin{equation}\label{eqn:weak1}
\| T(f)\|_{L_{l^s}^\alpha (\mu_E)} \lesssim  \mu(E)^{\frac{q-\alpha}{q\alpha}}\|f\|_{L^p(\sigma)}.
\end{equation}
By duality, \eqref{eqn:weak1} is equivalent to
\begin{equation*}
\| T^*(g)\|_{L^{p'}(\sigma)} \lesssim \mu(E)^{\frac{q-\alpha}{q\alpha}}\|g\|_{L_{l^{s'}}^{\alpha'} (\mu_E)}.
\end{equation*}

With these observations, Theorem \ref{thm:qp} gives that the following result holds:

\begin{thm}
Let $1<q<p$ . Then \eqref{eqn:weak0} holds if and only if for a fixed $1<\alpha<q$ and for any measurable $E\subset \R^n$ such that $\mu(E)>0$ the following two conditions hold:
\begin{enumerate}
\item\label{enum:weakqp1} There exists $C_1>0$, such that for any $g=(g_Q)_{Q\in{\mathcal D}}$, $g_Q\in\R$, $Q\in\R$,  and any ${\mathcal G}\subset {\mathcal D}$ $\mu_{|E}$-sparse, if $T^*_{|E}$ is the adjoint operator of $T_{|E}$ , then
$$\left\|\left\{  \frac{\| T^*_{|E}(g_G)\|_{L^{p'}(\sigma)}}{\mu_{|E}(G)^{\frac1{q'}}\|g_G\|_{L_{s'}^\infty(\mu_{|E})}} \right\}_{G\in{\mathcal G}}\right\|_{l^{r_\alpha}({\mathcal G})}\lesssim C_1\mu(E)^{\frac{q-\alpha}{q\alpha}},$$ where $1/r_\alpha=1/\alpha-1/p$.

\item\label{enum:weakqp2} There exists $C_2>0$, such that for any ${\mathcal F}$ $\sigma$-sparse, and any $(\beta_F)_{F\in{\mathcal F}}$, $\beta_F\geq 0$, 
$$ 
\left\| \sum_{F\in{\mathcal F}}\frac{\beta_F}{\sigma(F)^{1/p}}  T(\chi_F) \right\|_{L_{l^\alpha}^q(\mu_{|E})}
\leq C_2 \mu(E)^{\frac{q-\alpha}{q\alpha}}\left(  \sum_{F\in{\mathcal F}}\beta_F^p\right)^{1/p}.$$
\end{enumerate}

In addition, if $C, C_i$, $i=1,2$ are the smallest constants in \eqref{eqn:weak0}, \eqref{enum:weakqp1} and \eqref{enum:weakqp2} respectively, we have that $C\approx C_1+C_2$.
\end{thm}

\section{Characterizations based on reduction to $q=1$}\label{reduction}

A key argument used in this section is the equivalence between "sparse coefficients" and "Carleson coefficients" associated to the dyadic system ${\mathcal D}$ and the measure $\mu$ (see Theorem \ref{thm:dor}). An essential condition for this equivalence is that the measure $\mu$ has no point masses. The first theorem will give, for a measure $\mu$ with no point masses, a "canonical" way to choose sets with a prescribed mass. Of course, this choice is not unique, but it will be convenient to have a canonical choice in order to prove Theorem \ref{thm:dor}.

 The proof of the following result will based on an induction process. In this process we will consider cubes $Q$ that contain some "faces". In the first step we fix an order on the family  of the "faces" of $Q$, that is the $n-1$-dimensional cubes that are in $\partial Q$, that we may assume pairwise disjoint, subtracting some of the "edges". This procedure to order the corresponding family of lower dimensional "faces" will be continued in an analogous way in the successive generations.

\begin{thm}\label{lem:canonical}
Let $\mu$ be a positive locally finite measure on $\R^n$ with no point masses. Let $Q$ an  $n$-dimensional cube in $\R^n$ which contains some of its "faces" and $0<m\leq \mu(Q)$. Then there exists an set $H_Q=H_Q(m)\subset Q$ such that $\mu(H_Q)=m$. This set can be chosen  "canonically" in such a way that if $0<m_1\leq m_2\leq \mu(Q)$, the corresponding "canonical" sets $H_Q(m_1)$ and $H_Q(m_2)$ satisfy that $H_Q(m_1)\subset H_Q(m_2)$.
\end{thm}
\begin{proof}

Let $Q$ an $n$-dimensional cube in $\R^n$. Let $x(Q)$ be the center of $Q$ and $0<m\leq \mu(Q)$. Let $Q_n(t)$ is a $t$-homotetic cube in $Q$ with the same center of $Q$, that is, $Q_n(t)= x(Q)+ t(Q-x(Q))$, $0\leq t\leq 1$. 
Let $f_n$ be the function defined on $[0,1]$ by $f_n(t):=\mu( Q_n(t))$. The function $f_n$ is a non-decreasing function with $f(0)=0$ since $x(Q)$ is not a point mass and $f(1)=\mu(Q)$.

Let $t_n^0=\sup\{t\in[0,1]\,;\, \mu(Q_n(t))< m\}$. We first observe that the fact that the measure $\mu$ has no point masses, gives that $t_n^0>0$.

There are two possibilities: 

\begin{enumerate}\item\label{enum:possib1} $f_n$ is continuous at $t_n^0$.
\item\label{enum:possib2} $f_n$ has a jump discontinuity at $t_n^0$.
\end{enumerate}

Assume that \eqref{enum:possib1} holds. If $t_n^0<1$, 
 we  have 
$$m\geq \lim_{t\rightarrow (t_n^0)^-} f_n(t)= \mu((Q_n(t_n^0))^\circ)= \lim_{t\rightarrow (t_n^0)^+} f_n(t)= \mu(\overline{Q_n(t_n^0)}) \geq m.$$
Then we choose the set 
$H_Q= (Q_n(t_n^0))^\circ$ and $\mu(H_Q)=m$.

 If $t_n^0=1$, then 
$$\lim_{t\rightarrow 1^-} f_n(t)=\mu((Q)^\circ)=\mu(Q)=m$$
and we also take $H_Q=(Q)^\circ$.

Assume now that \eqref{enum:possib2} holds. If $t_n^0<1$  we have 

$$\alpha:=\mu((Q_n(t_n^0))^\circ)=\lim_{t\rightarrow (t_n^0)^-} f_n(t) < \lim_{t\rightarrow (t_n^0)^+} f_n(t)=\mu(\overline{Q_n(t_n^0)}):=\beta.
$$
Hence $\mu(\partial Q_n(t_n^0))=\beta-\alpha$.

 If $t_n^0=1$, we replace $\partial Q_n(t_n^0)$ by $Q_n(1)\setminus (Q_n(1))^\circ=Q\setminus (Q)^\circ$.
From now on, we will assume that $t_n^0<1$ with the obvious changes for $t_n^0=1$.

We have that $\alpha\leq m\leq \beta$.
 Then $m-\alpha\in[0,\beta-\alpha]$ and we want to choose in a "canonical" way a measurable set in $\partial Q_n(t_n^0)$ of measure $m-\alpha$.

The set $\partial Q_n(t_n^0)$ can be identified to a finite union of $(n-1)$-dimensional cubes that we have made pairwise disjoints subtracting some edges. We order them in the fixed way given before,  $Q_{n-1}^1$, \dots, $Q_{n-1}^{i_n}$. We also consider the center of each of these cubes, $x(Q_{n-1}^{i})$, $1\leq i\leq i_n$.

Take $i$ the lowest index such that
$$
\mu(Q_{n-1}^1)+\cdots+\mu(Q_{n-1}^{i-1})\leq m-\alpha < \mu(Q_{n-1}^1)+\cdots+\mu(Q_{n-1}^{i}).
$$

If $\mu(Q_{n-1}^1)+\cdots+\mu(Q_{n-1}^{i-1})=m-\alpha$, we will choose as "canonical" set $H_Q=(Q_n(t_n^0))^\circ\cup\left(\cup_{j=1}^{i-1} Q_{n-1}^j\right) $.

If not, we continue iterating. We have that $m-\alpha-(\mu(Q_{n-1}^1)+\cdots+\mu(Q_{n-1}^{i-1}))\in (0, \mu(Q_{n-1}^i))$. We consider  the  function $f_{n-1}^i(t):=\mu( Q_{n-1}^i(t))$, where $Q_{n-1}^i(t)$ denotes the $t$-homotetic cube of $Q_{n-1}^{i}$ with respect to the center $x(Q_{n-1}^{i})$. Let $t_{n-1}^{i}= \sup\{t\in[0,1]\,;\, \mu( Q_{n-1}^i(t))< m-\alpha-(\mu(Q_{n-1}^1)+\cdots+\mu(Q_{n-1}^{i-1}))\}$. If the function $f_{n-1}^i$ is continuous at $t_{n-1}^{i}$, then we consider as "canonical" set $H_Q=(Q_n(t_n^0))^\circ\cup \left(\cup_{j=1}^{i-1} Q_{n-1}^j\right)\cup (Q_{n-1}^i(t_{n-1}^{i}))^\circ$ and we are done. If not, we can continue iterating. If this iteration stops at some stage, then we are done.  

We next show that the iteration has to stop at some stage. Assume that this is not the case. Then, finally, we can find a line segment $Q_1$ and a function $f_1(t):=\mu(Q_1(t))$ such that the function $f_1$ is discontinuous at some  $t$, which is equivalent to saying that there is a point $x\in Q_1$ such that $\mu(\{x\})>0$. And this contradicts the assumption that $\mu$ has no point masses. Therefore, the iteration has to stop. Observe that $H_Q$ is always a subset of $Q$, and it is a cube with part of its boundary. We will call such set a "canonical" extended cube.

Finally, observe that the "canonical" method chosen to construct the sets $H_Q$ give that if $0<m_1<m_2<\mu(Q)$, then the corresponding "canonical" extended cubes $H_Q(m_1)$ and $H_Q(m_2)$ satisfy $H_Q(m_1)\subset H_Q(m_2)$.
\end{proof}
\begin{rem}\label{rem:canonical}
We remark that if $F\subset Q$  is  measurable and $m\in [0,\mu(Q\setminus F)]$, the same method beginning with $t_0=\sup\{t\in[0,1]\,;\, \mu(Q(t)\setminus F)< m\}$ and proceeding analogously subtracting the set $F$ in the previous arguments   permits  to obtain  an extended "canonical"  cube $H_{Q,F}$ in $Q$ such that $\mu(H_{Q,F}\setminus F)= m$. With this procedure, if $F\subset G$, and $m\in [0,\mu(Q\setminus G)]$, the corresponding extended cubes satisfy $H_{Q,F}\subset H_{Q,G}$.
\end{rem}

\begin{cor}\label{lem:pointmasses}
Let $\mu$ be a positive locally finite Borel measure on $\R^n$. Then the following assertions are equivalent:
\begin{enumerate}
\item\label{item:pointmasses1} The measure $\mu$ has no point masses.
\item\label{item:pointmasses2} For each measurable set $A$ and for every $m\in[0,\mu(A)]$, there exists a measurable subset $H\subset A$ such that $\mu(H)=m$.
\end{enumerate}
\end{cor}
\begin{proof}
It is clear that \eqref{item:pointmasses2} implies \eqref{item:pointmasses1}. For the proof of \eqref{item:pointmasses1} implies 
\eqref{item:pointmasses2}. Let $A$ be a measurable set. By a limiting argument, we may assume that the set $A$ is contained in a cube.  Theorem \ref{lem:canonical} applied to the measure $\mu_{|A}$ finishes the proof.
\end{proof}

The following theorem was stated in \cite{verbitsky1}, Corollary 2, using a result of \cite{dor} under some implicit conditions on the involved measures. Here we will give a direct proof with the only assumption that the positive locally finite Borel measure has no point masses. We will follow the  arguments used in Lemma 6.3 in \cite{lernernazarov}, where it is given an equivalence between sparse and Carleson families of dyadic sets, established with respect the Lebesgue measure. Although we will follow closely the same arguments, we believe that can be convenient for the reader to give a sketch of the proof.

\begin{thm}\label{thm:dor}
Let $\mu$ be a locally finite measure on $\R^n$ without point masses. Let $(\lambda_Q)_Q$ be non-negative reals and $C>0$. Then the following assertions are equivalent:
\begin{enumerate}
\item\label{item:dor1}
The coefficients $(\lambda_Q)_Q$ satisfy the Carleson condition with constant $C$, that is,  for every dyadic cube $P\in{\mathcal D}$,
$$\sum_{Q\subset P} \lambda_Q\leq C \mu(P).$$
\item\label{item:dor2} There exist pairwise disjoint sets $E_Q\subset Q$ such that for any $Q\in{\mathcal D}$,
$$\lambda_Q\leq C\mu(E_Q).$$

\end{enumerate}
\end{thm}
\begin{proof}

It is clear that \eqref{item:dor2} implies \eqref{item:dor1}. Assume that \eqref{item:dor1} holds, that is,
$$\sum_{Q\subset P} \frac{\lambda_Q}{C}\leq \mu(P).$$
This gives, in particular, that
\begin{equation}\label{eqn:carleson1} \frac{\lambda_Q}{C} \leq \mu(Q)\end{equation}
and
\begin{equation}\label{eqn:carleson2}\frac{\lambda_P}{C} \leq \mu(P)- \sum_{Q\subsetneq P}\frac{\lambda_Q}{C}.\end{equation}
One way of interpreting the meaning of the key observation \eqref{eqn:carleson2} is to say that if from $\mu(P)$  we take off ${\displaystyle \frac1{C}\lambda_Q}$ of mass of every strict cube $Q\subsetneq P$, then there is still at least ${\displaystyle \frac1{C}\lambda_P}$ of mass left in $P$.

Denote for each $k\geq 1$, ${\mathcal D}_k$ the family of dyadic cubes of generation $k$, that is, cubes with side length $1/2^k$. If the cubes in ${\mathcal D}$ such that $\lambda_Q\neq0$ are of size bounded from below by a positive number, that is, if there exists $k_0$ such that for any $k\geq k_0$ and for any $Q\in{\mathcal D_k}$, $\lambda_Q=0$, we can proceed by an argument starting from bottom up and considering only cubes in $\cup_{i\leq k_0}{\mathcal D}_i$ . For each $Q\in{\mathcal D}_{k_0}$, by Lemma \ref{lem:pointmasses},  choose any set $E_Q \subset Q$ with the mass ${\displaystyle \mu(E_Q)=\frac{\lambda_Q}{C}}$. 
  The chosen sets $E_Q$ of course are pairwise disjoint.  For the cubes in the next generation ${\mathcal D}_{k_0-1}$, if $P\in {\mathcal D}_{k_0-1}$, by \eqref{eqn:carleson2}, $\displaystyle{\mu(P\setminus \cup_{Q\subsetneq P}E_Q)\geq \frac1{C}\lambda_P}$. Hence, corollary \ref{lem:pointmasses} applied to the set $P\setminus \cup_{Q\subsetneq P} E_Q$, gives that there exists a measurable set $E_P$,  disjoint with any $E_Q$ with $Q\subsetneq P$ and such that $\displaystyle{\mu(E_P)=\frac1{C}\lambda_P}$. Iterating this process, we obtain \eqref{item:dor2} in this particular case. We observe that in this situation, the "canonical" choices of the extended cubes $E_Q$ are irrelevant.
	
	For the general case, we will follow a limiting argument  as in \cite{lernernazarov} as well as the "canonical"  extended cubes with prescribed mass given in Theorem \ref{lem:canonical}.
	
	Let $K\in\Z$ be fixed, and let $Q\in \cup_{k\leq K}{\mathcal D}_k$, $k\leq K$. We define  sets $\hat{E}_Q^K$ 
	"canonically" by induction on $k$, satisfying the following properties:
	\begin{enumerate}
	\item[(a)]\label{item:condition1} $\hat{E}_Q^K\subset Q$.
	\item[(b)]\label{item:condition2} If we define 
	$$F_Q^K:=\bigcup_{\substack{R\subsetneq Q\\ R\in \cup_{k+1\leq i \leq K}{\mathcal D}_i}}\hat{E}_R^K,$$  then 
	$\mu(\hat{E}_Q^K\setminus F_Q^K)=\frac1{C}\lambda_Q$.
	\end{enumerate}
	
Indeed, let $Q\in{\mathcal D}_K$. Since $\frac1{C}\lambda_Q\leq \mu(Q)$, if $\hat{E}_Q^K\in Q$ is extended cube given in Theorem \ref{lem:canonical}, we have that
	\begin{equation}\label{eqn:canon1}
	\mu(\hat{E}_Q^K)=\frac1{C}\lambda_Q.
	\end{equation}
	If we consider $F_Q^K=\emptyset$, we have the construction for the first generation.
	
	Assume now that we have defined  $\hat{E}_R^K\subset R$  for $R\in \cup_{k+1\leq i \leq K}{\mathcal D}_i$ satisfying (a) and (b). Let $Q\in {\mathcal D}_k$ and let $H_{Q,F_Q^K}^K\in Q$ be the "canonical" extended cube (see Remark \ref{rem:canonical}) satisfying
	$$ \mu( H_{Q,F_Q^K}^K\setminus F_Q^K)= \frac1{C}\lambda_Q.$$
	
	Observe that this set exists, since by the induction hypothesis \eqref{item:condition2} and inequality (b), $\mu(Q\setminus F_Q^K)\geq \frac1{C}\lambda_Q$.
	
	 We define
	$$
	 \hat{E}_Q^K:= F_Q^K\cup H_{Q,F_Q^K}^K.
	$$
	Then 
$$
\mu(\hat{E}_Q^K\setminus F_Q^K)= \mu(  H_{Q,F_Q^K}^K\setminus F_Q^K)=\frac1{C}\lambda_Q 
$$
	and we have completed the induction.

	Let us check that with this construction, for any $Q\in{\mathcal D}_k$, $k\leq K$, $\hat{E}_Q^K\subset \hat{E}_Q^{K+1}$. We proceed inductively on $k$. Indeed, if $Q\in{\mathcal D}_K$, $\hat{E}_Q^K$ is the canonical set included in $Q$ and such that $\mu(\hat{E}_Q^K)=\frac1{C}\lambda_Q$  and   $\hat{E}_Q^{K+1}= F_Q^{K+1}\cup H_{Q,F_Q^{K+1}}^{K+1}$. In addition, Remark \ref{rem:canonical} gives that $\hat{E}_Q^K\subset H_{Q,F_Q^{K+1}}^{K+1}$.

	Assume now that $\hat{E}_Q^K\subset \hat{E}_Q^{K+1}$ for $Q\in\cup_{k< i\leq K}{\mathcal D}_k$. Let $Q\in{\mathcal D}_k$. By hypothesis we have that $ F_Q^K\subset F_Q^{K+1}$ and, consequently, by Remark \ref{rem:canonical}, we have that $H_{Q,F_Q^K}^K\subset H_{Q,F_Q^{K}}^{K+1}$. Thus, $\hat{E}_Q^K\subset \hat{E}_Q^{K+1}$.
	
	Let $Q\in{\mathcal D}_k$. We  define
	$$\hat{E}_Q:=\lim_K \hat{E}_Q^K= \cup_{K\geq k} \hat{E}_Q^K\subset Q.$$
Since the sequences of sets $(F_Q^K)_K$ and $(\hat{E}_Q^K)_K$ are non-decreasing and $F_Q^K\subset \hat{E}_Q^K$, there exists $$\lim_{K\rightarrow\infty} \left(\hat{E}_Q^K\setminus F_Q^K \right)$$ and coincide with $\left(\cup_{K\geq k} \hat{E}_Q^K\setminus\cup_{K\geq k} F_Q^K\right)$. 

We choose the sets $E_Q$ of \eqref{item:dor2} as
$$E_Q :=\lim_{K\rightarrow\infty} \left(\hat{E}_Q^K\setminus F_Q^K \right).$$

  Since for each $K$, $\mu(\hat{E}_Q^K\setminus F_Q^K     )=\frac1{C}\lambda_Q$, we deduce that $\mu(E_Q)=\frac1{C}\lambda_Q$.
Next, 
$$E_Q= \left(\cup_{K\geq k} \hat{E}_Q^K\setminus\cup_{K\geq k} F_Q^K\right)=
 \hat{E}_Q\setminus  \left( \cup_{R\subsetneq Q} \hat{E}_R\right), $$
and consequently, the sets $E_Q$, which are subsets of $\hat{E}_Q$,  are pairwise disjoint.

\end{proof}
\begin{cor}\label{cor:lambda_1=lambda_2}
Let $\mu$ be a positive locally finite Borel measure on $\R^n$ with no point masses. Let $(\lambda_Q)_Q$ be non-negative real numbers. Then, if we denote $$\Lambda_1:=\sup_P \frac{\sum_{Q\subset P}\lambda_Q}{\mu(P)}$$ and $$\Lambda_2:= \mathop{\inf_{E_Q\subset Q }}_{{\rm pairwise\,\,\,disjoints}}\sup_Q \frac{\lambda_Q }{\mu(E_Q)},$$
we have that $\Lambda_1=\Lambda_2$.
\end{cor}
\begin{proof}
Since for each $P$, $\sum_{Q \subset P}\lambda_Q\leq \Lambda_1\mu(P)$, the  equivalence  between \eqref{item:dor1} and \eqref{item:dor2} of Theorem \ref{thm:dor} shows that there exists $E_Q\subset Q$ pairwise disjoint such that 
$\lambda_Q\leq \Lambda_1\mu(E_Q)$, and hence $\Lambda_2\leq \Lambda_1$. Reciprocally, if $\Lambda_2<C$, the equivalence  between \eqref{item:dor1} and \eqref{item:dor2} shows that for each $P$, $\sum_{Q\subset P} \lambda_Q\leq C\mu(P)$. Consequently, $\Lambda_1\leq C$ and then $\Lambda_1\leq \Lambda_2$.
\end{proof}
\begin{rem}\label{rem:necessitynopointmasses}
If the measure $\mu$ has point masses, the above Theorem \ref{thm:dor} may fail. Indeed, take two nested dyadic cubes $Q_1\subset Q_2$, two non-zero coefficients, $\lambda_{Q_1}$ and $\lambda_{Q_2}$ and a point mass $\mu$ contained in both the cubes. Then, clearly, the Carleson condition \eqref{item:dor1} holds, but the condition \eqref{item:dor2}
fails since one can not divide the mass point.
      \end{rem}

We follow with  a result that will be used in the proof of Theorem \ref{thm:reduction}.  For the Lebesgue measure and the case $1<s<\infty$ it was proved in Corollary 5.12 in \cite{frazierjawerth}  and, in the general setting, with a different proof in an unpublished work by I.E. Verbitsky. For a sake of completeness, we include a sketch of the proof.

\begin{lem}\label{lem:lemma00} Let $1< s <\infty$ and $\mu$ a locally finite measure on $\R^n$. Let $\Lambda=(\lambda_Q)_Q \subset [0,\infty)$ be a sequence such that $\lambda_Q=0$ if $\sigma(Q)=0$. Define  
$$A_1(\Lambda)= \left\|\left(\sum_{Q\in {\mathcal D}}\lambda_Q^{ s} \chi_Q\right)^\frac1{{s}} \right\|_{L^1(\mu)},$$
and
$$
A_2(\Lambda)=\sup_{(\alpha_Q)_{Q\in{\mathcal D}}} \frac{\sum_{Q\in{\mathcal D}}\lambda_Q \alpha_Q}{ 
\sup_{\substack{P\in{\mathcal D};\, \mu(P)\neq0}} \left(
\frac1{\mu(P)} \sum_{Q\subset P} \left( \frac{\alpha_Q}{\mu(Q)}\right)^{{ s}'}\mu(Q)\right)^{1/{{ s}'}}},
$$
where the supremum is taken over all sequences $(\alpha_Q)_Q\subset[0,\infty)$ such that $\alpha_Q=0$ if $\mu(Q)=0$.
Then there exists $C=C(s)$ such that
$$
CA_2(\Lambda)\leq A_1(\Lambda)\leq A_2(\Lambda).
$$
\end{lem}
\begin{proof}

The estimate $CA_2(\Lambda)\leq  A_1(\Lambda)$ was proved in Theorem 4.b in \cite{verbitsky1} (consider $p=r=1<q$ in that Theorem). For the reverse estimate,
\begin{align*}
&A_1(\Lambda)=\int_{\R^n} \left( \sum_{Q\in{\mathcal D}} \lambda_Q^s\chi_Q\right)^{1/s}d\mu\\&=\int_{\R^n} \left( \sum_{Q\in{\mathcal D}} \lambda_Q^s\chi_Q\right)^{1-1/s'}d\mu=\sum_Q \lambda_Q^s \int_Q \frac{d\mu}{\left( \sum_{R\in{\mathcal D}} \lambda_R^s \chi_R\right)^{1/s'}}\\&
=\sum_{Q\in{\mathcal D}} \lambda_Q \alpha_Q,
\end{align*}
where
$$\alpha_Q= \lambda_Q^{s-1}\int_Q \frac{d\mu}{\left( \sum_{R\in{\mathcal D}} \lambda_R^s \chi_R\right)^{1/s'}},\quad Q\in{\mathcal D}.
$$
Next, H\"older's inequality gives that for every $P\in{\mathcal D}$,
\begin{align*}
& \sum_{Q\subset P}  \left( \frac{\alpha_Q}{\mu(Q)}\right)^{s'}\mu(Q)\\&=\sum_{Q\subset P} \alpha_Q^{s'}\mu(Q)^{1-s'} =\sum_{Q\subset P} \lambda_Q^s \mu(Q)^{1-s'} \left( 
\int_Q \frac{d\mu}{\left( \sum_{R\in{\mathcal D}} \lambda_R^s \chi_R\right)^{1/s'}}\right)^{s'}\\&
\leq \sum_{Q\subset P} \lambda_Q^s \int_Q \frac{d\mu}{\sum_{R\in{\mathcal D}}\lambda_R^s \chi_R}\leq \int_P d\mu=\mu(P).
\end{align*}
Consequently, for the set of chosen $(\alpha_Q)_Q$ we have that
\begin{equation*}\begin{split}&
A_1(\Lambda)=\sum_{Q\in{\mathcal D}} \lambda_Q\alpha_Q\\& \leq \frac{\sum_{Q\in{\mathcal D}} \lambda_Q \alpha_Q}{\sup_{P\in{\mathcal D}}\left( \frac1{\mu(P)} \sum_{Q\subset P} \alpha_Q^{s'} \mu(Q)^{-1/(s-1)}\right)^{1/s'}} \leq A_2(\Lambda).
\end{split}\end{equation*}

\end{proof}
\begin{rem}\label{rem:remarks=1}
This lemma is also true for $s=1$ and $s=\infty$, but we do not include the proof, since it will not be necessary for our purposes.
\end{rem}
As a consequence of Theorem \ref{thm:dor} and Lemma \ref{lem:lemma00}, we have the following result.

\begin{lem}\label{lem:dorverbitsky}
Let $(b_Q)_Q$ be a sequence of non-negative real numbers. Let $0<q<\infty$ and $q\leq s\leq\infty$. Let $\mu$ be a positive locally finite Borel measure with no point masses. Then the following assertions are equivalent:
\begin{enumerate}
\item\label{item:dorverbitsky1}  $\|\left( \sum_Q b_Q^s\chi_Q \right)^{1/s} \|_{L^q(\mu)} \lesssim C.$
\item\label{item:dorverbitsky2} Define $\tilde{s}:=s/q$. Then for every collection $(E_Q)_Q$ of pairwise disjoint sets with $E_Q\subset Q$,
$$
\left( \sum_Q b_Q^q \mu(E_Q)^{1/\tilde{s}'} \mu(Q)^{1/{\tilde{s}}}\right)^{1/q} \lesssim C,
$$

where here $\lesssim$ means that the constants involved  may depend on $s$, but not on the sequence $(b_Q)_Q$.
\end{enumerate}
\end{lem}
\begin{proof}
Notice that the endpoint cases are trivial: for $s=q$ the expressions coincide, and for $s=\infty$ (in which case the summation is interpreted as the supremum), the assertion is clear by linearizing the supremum (we can write $\sup_Q b_Q\chi_Q=\sum_Q b_Q \chi_{E_Q}$ for some pairwise disjoint sets $E_Q\subset Q$).

Assume now that $q<s<\infty$. Taking $\tilde{s}:=s/q$, \eqref{item:dorverbitsky1} can be rewritten as

$$\|\left( \sum_Q (b_Q^q)^{\tilde{s}}\chi_Q \right)^{1/\tilde{s}} \|_{L^1(\mu)} \lesssim C^q.$$

By Lemma \ref{lem:lemma00}, \eqref{item:dorverbitsky1} holds if and only if for every sequence $(\alpha_Q)_Q$ of none-negative reals such that $\alpha_Q=0$ if $\mu(Q)=0$ we have
$$
\sum_Q b_Q^q\alpha_Q \lesssim C^q \sup_P \left(  \frac1{\mu(P)} \sum_{Q\subset P} \left( \frac{\alpha_Q}{\mu(Q)}\right)^{\tilde{s}'} \mu(Q)\right)^{1/{\tilde{s'}}}.
$$

Next, Corollary \ref{cor:lambda_1=lambda_2} gives that the above estimate holds if and only if for every family $(E_Q)_Q$ of pairwise disjoint sets, $E_Q\subset Q$, we have that
\begin{align*}&
\sum_Q b_Q^q\alpha_Q \\&\lesssim C^q \sup_Q \left(  \frac{\left( \frac{\alpha_Q}{\mu(Q)}\right)^{\tilde{s}'}\mu(Q)}{\mu(E_Q)}\right)^{1/\tilde{s'}}=  C^q \sup_Q \frac{\alpha_Q}{\mu(Q)^{1/{\tilde s}}\mu(E_Q)^{1/{{\tilde s}'}}}.
\end{align*}

We define
$$\beta_Q:=\frac{\alpha_Q}{\mu(Q)^{1/{\tilde s}}\mu(E_Q)^{1/{{\tilde s}'}} }.$$

We then have,
$$
\sum_Q b_Q^q \mu(E_Q)^{1/{\tilde{s}}'} \mu(Q)^{1/{\tilde{s}}}\beta_Q \lesssim C^q\sup_Q \beta_Q,
$$
which is equivalent to,
$$
\left( \sum_Q b_Q^q \mu(E_Q)^{1/{\tilde{s}}'} \mu(Q)^{1/{\tilde{s}}} \right)^q\lesssim C.
$$

\end{proof}

\subsection{Proof of Theorem \ref{thm:reduction}}

\begin{proof}
The condition \eqref{eqn:qpduality} says that
$$\|\left( \sum_Q \lambda_Q^s\left( \int_Q fd\sigma\right)^s\chi_Q \right)^{1/s} \|_{L^q(\mu)} \leq C\|f\|_{L^p(\sigma)}.$$

Since $q\leq s\leq\infty$, by Lemma \ref{lem:dorverbitsky} this estimate holds, if and only if for every collection $(E_Q)_Q$ of pairwise disjoint sets with $E_Q\subset Q$, and denoting ${\tilde s}=s/q$,
$$
\left( \sum_Q \lambda_Q^q \left(\int_Qfd\sigma \right)^q\mu(E_Q)^{1/{\tilde{s}}'} \mu(Q)^{1/{\tilde{s}}}\right)^{1/q} \leq C\|f\|_{L^p(\sigma)}.
$$
By Lemma \ref{lemma1},  scaling the index, since $q<p$, the above is equivalent  is equivalent to
$$\sum_Q\lambda_Q^q \left(\frac{\int_Qfd\sigma}{\sigma(Q)} \right) \sigma(Q)^q\mu(E_Q)^{1/{\tilde{s}}'}\mu(Q)^{1/{\tilde{s}} }\leq C^q \|f\|_{L^{p/q}(\sigma)},$$
estimate that can be rewritten as
$$\int_{\R^n} \left(\sum_Q \frac{(\lambda_Q \sigma(Q))^q\mu(E_Q)^{1/{\tilde{s}}'}\mu(Q)^{1/{\tilde{s}}}\chi_Q}{\sigma(Q)} \right) fd\sigma\leq C^q \|f\|_{L^{p/q}(\sigma)},$$

which by duality is equivalent to
\begin{equation}\label{eqn:equationdual1}
\left\| \left(\sum_Q \frac{(\lambda_Q \sigma(Q))^q\mu(E_Q)^{1/{\tilde{s}}'}\mu(Q)^{1/{\tilde{s}}}\chi_Q}{\sigma(Q)} \right)    \right\|_{L^{(p/q)'}(\sigma)}\leq C^q.
\end{equation}

\end{proof}

\begin{rem}
Observe that if $s=q$ (equivalently, ${\tilde s}=1$), \eqref{eqn:equationdual1} says simply that
$$\left\| \sum_{Q\in{\mathcal D}} \frac{(\lambda_Q\sigma(Q))^q\mu(Q)}{\sigma(Q)}\chi_Q \right\|_{L^{{\tilde p}'}(\sigma)}\leq C,$$
which is the trivial condition for ${\tilde s}=1$ and ${\tilde p}>1$ that can be obtained directly by duality.
In the other extreme case $s=\infty$, which corresponds to the dyadic maximal function, the condition is that for every subcollection of pairwise disjoint sets $E_Q\subset Q$
$$\left\| \sum_{Q\in{\mathcal D}} \frac{(\lambda_Q\sigma(Q))^q\mu(E_Q)}{\sigma(Q)}\chi_Q \right\|_{L^{{\tilde p}'}(\sigma)}\leq C.$$\end{rem}

\begin{rem}\label{rem:rem2}
We observe that if $T$ is such that there exists a subcollection of cubes ${\mathcal S}$  $\mu$-sparse and $T(f)=\left(\lambda_Q \left(\int_Q fd\sigma\right)\chi_Q \right)_{Q\in {\mathcal S}}$.
For this operator Theorem \ref{thm:reduction} gives that $T$ is bounded from $L^p(\sigma)$ to $L_{{l^s}}^q(\mu)$ if and only if
$$\left\| \sum_{Q\in{\mathcal S}} \frac{(\lambda_Q\sigma(Q))^q\mu(Q)}{\sigma(Q)}\chi_Q \right\|_{L^{{\tilde p}'}(\sigma)}\leq C. $$
Observe that this condition is independent of $s$.
\end{rem}

\section{Application to Wolff potentials}    

If $\alpha,s>0$, $0<\alpha <n$, we consider the dyadic Wolff's potential defined by
\begin{equation*}
{\mathcal W}_{\alpha,\,s}^{\mathcal D}(f)(x)=\sum_{Q \in{\mathcal D}} \left( \frac{\int_Q fdx}{|Q|^{1-\alpha/n}}\right)^s\chi_Q(x),
\end{equation*}
where $|Q|$ denotes the Lebesgue measure.

If $1<q<p$ and $q\leq s$, the question that we want to consider is the following: which are the measures $\mu,\sigma$ such that
\begin{equation}\label{eqn:wolffestimate}
\|{\mathcal W}_{\alpha,\,s}^{\mathcal D}(f)^{1/s}\|_{L^q(\mu)} \lesssim \|f\|_{L^p(\sigma)}?
\end{equation}
In \cite{cascanteortegaverbitsky} it is given the relationship of Wolff's potential with the Riesz potentials and some applications.

Theorem \ref{thm:reduction} gives that, if $\mu$ has no point masses,  \eqref{eqn:wolffestimate} holds if and only if there exists $C>0$ such that for any $(E_Q)$ pairwise disjoint and such that $E_Q\subset Q$,
$$\left\| \sum_{Q\in{\mathcal D}} \frac{(|Q|^{\alpha/n-1}\sigma(Q))^q\mu(Q)^{1/{\tilde{s}}} \mu(E_Q)^{1/{\tilde s}'}}{\sigma(Q)}\chi_Q \right\|_{L^{{\tilde p}'}(\sigma)}\leq C, $$
where ${\tilde s}=s/q$ and ${\tilde p}=p/q$.

For the particular case where $\mu$ is an $A_\infty$ weight,  we have:
\begin{cor}
let $\mu$ be an $A_\infty$-weight. Estimate \eqref{eqn:wolffestimate} holds if and only if there exists a constant $C>0$ such that for any subcollection ${\mathcal S}$ of dyadic cubes, $\mu$-sparse,
$$\left\| \sum_{Q\in{\mathcal S}} \frac{(|Q|^{\alpha/n-1}\sigma(Q))^q\mu(Q)}{\sigma(Q)}\chi_Q \right\|_{L^{{\tilde p}'}(\sigma)}\leq C. $$
\end{cor}

\begin{proof}

We first observe that if $\mu$ is an $A_\infty$ weight, then any subcollection ${\mathcal S}$ sparse with respect to the Lebesgue measure is also $\mu$-sparse. Indeed, recall that a measure $\mu$ is in $A_\infty$ if for each $0<\gamma<1$ there exists $0<\delta<1$ such that for each $Q$ and any subset $E\subset Q$, $|E|\subset \gamma|Q|$, then $\mu(E)\leq \delta\mu(Q)$.
 Then the proof of the corollary  is an immediate a consequence of Remark \eqref{rem:rem2}  and the following lemma. 

\begin{lem}
Let ${\mathcal D}$ be a dyadic grid in $\R^n$ and a non-negative compactly supported integrable function $f$. There exists a subcollection ${\mathcal S}$ of dyadic cubes, sparse with respect to the Lebesgue measure such that
 
$$ {\mathcal W}_{\alpha,\,s}^{\mathcal D}(f)\lesssim {\mathcal W}_{\alpha,\,s}^{\mathcal S}(f).$$
Of course, for every subcollection ${\mathcal S}_1$ of dyadic cubes, sparse with respect to the Lebesgue measure, we also have
$${\mathcal W}_{\alpha,\,s}^{{\mathcal S}_1}(f)\leq {\mathcal W}_{\alpha,\,s}^{\mathcal D}(f).$$
\end{lem}
\begin{proof}
The proof follows closely the proof  in Proposition 3.8 in \cite{cruz-uribe}, where a similar pointwise estimate is obtained for a fractional  operator, that is, $s=1$.
It gives a pointwise estimate of the dyadic Wolff's potential in terms of dyadic Wolff's potentials defined on subcollections of dyadic cubes, sparse with respect to the Lebesgue measure.
\end{proof}
\begin{rem}
Let $\mathcal{W}_{\alpha,\,s}(f)$ be the continuous Wolff potential given by
$$\mathcal{W}_{\alpha,\,s}(f)(x)=\int_0^{+\infty}\left(\frac{\int_{B(x,t)}fdx}{t^{n-\alpha }}\right)^s\frac{dt}{t},\qquad x\in\R^n.$$

We have that for any ${\mathcal D}$ dyadic grid on $\R^n$, ${\displaystyle {\mathcal W}_{\alpha,\,s}^{\mathcal D}(f) \lesssim {\mathcal W}_{\alpha,\,s}(f)}$ and, on the other hand, there exists ${\mathcal D}_i$, $i=1,\dots M$, families of dyadic grids on $\R^n$ such that ${\displaystyle {\mathcal W}_{\alpha,\,s}(f) \lesssim \sum_{i=1}^M {\mathcal W}_{\alpha,\,s}^{{\mathcal D}_i}(f)}$. Thus the above corollary gives a necessary and sufficient condition for the boundedness of $({\mathcal W}_{\alpha,\,s}(f))^{1/s}$ from $L^p(\sigma)$ to $L^q(\mu)$.
\end{rem}
\end{proof}


\begin{thebibliography}{COV1}

 


 \bibitem{cascanteortega} Cascante, C.; Ortega, J. M. On the boundedness of discrete Wolff potentials. Ann. Sc. Norm. Super. Pisa Cl. Sci.  \textbf{8} (2009), no. 2, 309–-331.


\bibitem{cascanteortegaverbitsky}
Cascante, C.; Ortega, J. M. and  Verbitsky, I.E.: Trace
inequalities of Sobolev type in the upper triangle case, 
  Proc. London Math. Soc., {\bf 80}, (2000), 391--414.


\bibitem{cascanteortegaverbitsky3}
Cascante, C.; Ortega, J. M. and  Verbitsky, I.E.: On $L^p-L^q$ inequalities,
     J.  London Math. Soc., {\bf 7}, (2006),  497--511.
		
		
		
		\bibitem{cruz-uribe} Cruz-Uribe, D.: Two weight norm inequalities for fractional integral operators ans commutators. Preprint 2014. ArXiv:1412.4157v1 [math.CA].
		

		
\bibitem{dor} Dor, L.E.: On projections in $L_1$.  Ann. of Math. \textbf{102 } (1975), 463-474. 

\bibitem{frazierjawerth} Frazier, M.; Jawerth, B.: A discrete transform and decomposition of distributional space. Journal of Functional Analysis {\bf 93} (1990), 34-170. 

\bibitem{hanninen1} H\"anninen, T.S.: Two weight inequality for vector-valued positive dyadic operators by parallel stopping cubes. Preprint 2014. arXiv:1404.6933 [math.CA].

 \bibitem{hanninen} H\"anninen, T.S.: Another proof of Scurry's characterization of a two weighted norm inequality for a sequence-valued positive dyadic operator. Preprint 2014. arXiv:1304.7759v1 [math.CA].
\bibitem{hanninenthesis} H\"anninen, T.S.: Dyadic analysis of integral operators: median oscillation decomposition and testing conditions. Thesis (Ph.D)-University of Helsinki (2015). Available at http://urn.fi/URN:ISBN:978-951-51-1393-1.
\bibitem{hanninenhytonen} H\"anninen, T.S.; Hyt\"onen, T.P.: Operator-valued dyadic shifts and the $T(1)$-Theorem. Preprint 2014. ArXiv:1412.0470 [math.CA].

\bibitem{hanninenhytonenli} H\"anninen, T.S.; Hyt\"onen, T.P.; Li, K.: Two-weight $L^p-L^q$ bounds for positive dyadic operators: unified approach to $p\leq q$ and $p>q$. Preprint 2014. ArXiv:1412.2593v1 [math.CA].
\bibitem{hytonen} Hyt\"onen, T.: The $A_2$ theorem:remarks and complements. In {\sl Harmonic analysis and partial differential equations}, Contemp. Math. \textbf{612} (2014), 91-106.

 \bibitem{kolmogorov}   Garcia-Cuerva, D. J. and  Rubio de Francia, J. L.:  Weighted inequalities and related topics, North-Holland Mathematics Studies 116 (North-Holland, Amsterdam, 1985).

 \bibitem{laceysawyeruriartetuero}  Lacey, M.T.; Sawyer, E.T.; Uriarte-Tuero, I.: Two weight inequalities for discrete positive operators. Preprint 2009. ArXiv:0911343704 [math.CA].

\bibitem{lernernazarov} Lerner, A.K. and Nazarov, F.: Intuitive dyadic calculus: the basics. Preprint 2015. Arxiv:1508.05639 [math.CA].
\bibitem{nazarovtreilvolberg} Nazarov, F.; Treil, S. and Volberg A.: The Bellman function and two-weight inequalities for Haar multipliers. J. Amer. Math. Soc. \textbf{12} (1999), 900-929.

\bibitem{sawyer} Sawyer, E.T.: A characterization of a two weight norm inequality for maximal operators. Studia Math. \textbf{75} (1982), 1-11.
 \bibitem{scurry} Scurry, J.: A characterization of two-weight inequalities for a vector-valued operator. Preprint 2010. arXiv:1007.3089 [math.CA].

\bibitem{tanaka}  Tanaka, H.: A characterization of two-weight trace inequalities for positive dyadic operators in the upper triangle case. Potential Anal. \textbf{41} (2014), no. 2, 487–-499.


\bibitem{treil} Treil, S.: A remark on two weight estimates for positive dyadic operators. Preprint 2012. arXiv:1201.1455 [math-CA].
 \bibitem{verbitsky1}
 Verbitsky,  I.E.:Imbedding and multiplier theorems for discrete 
Littlewood--Paley spaces, 
Pacific J. Math. {\bf 176} (1996), 529--556.



\bibitem{vuorinen} Vuorinen, E.: $L^p(\mu)\to L^q(\nu)$ characterizations for well localized operators. Preprint 2014. ArXiv1412.2117 [math.CA].





\end{thebibliography}
\end{document}